\documentclass[a4paper,11pt]{my-ijamas}

\usepackage{a4wide}

\theoremstyle{plain}
\newtheorem{problem}[theorem]{Problem}
\newtheorem{example}[theorem]{Example}

\usepackage{url}

\begin{document}

\title{Non-conservative Noether's theorem for fractional action-like
variational problems with intrinsic and observer times\thanks{The original publication is available at http://www.ceser.res.in/ijees.html,
Int. J. Ecol. Econ. Stat.,
Vol.~9, Nr~F07, 2007, pp.~74--82.}}

\author{\textbf{Gast\~{a}o S. F. Frederico$^1$ and Delfim F. M. Torres$^2$}}

\date{$^1$Department of Science and Technology\\
University of Cape Verde\\
Praia, Santiago -- Cape Verde\\
\url{gfrederico@mat.ua.pt}\\ [0.3cm]
$^2$Department of Mathematics\\
University of Aveiro\\
3810-193 Aveiro, Portugal\\
\url{delfim@ua.pt}}

\maketitle

%%%%%%%%%%%%%%%%%%%%%%%%%%%%%%%%%%%%%%%%%%%%%%%%%%%%%%

\begin{abstract}
\noindent \emph{We extend Noether's symmetry theorem to fractional
action-like variational problems with higher-order derivatives.}

\medskip

\noindent\textbf{Keywords:} fractional action-like variational
approach (FALVA), symmetry, conservation laws, Noether's theorem,
higher-order derivatives.

\medskip

\noindent\textbf{2000 Mathematics Subject Classification:} 49K05,
49S05, 70H33.

\end{abstract}

%%%%%%%%%%%%%%%%%%%%%%%%%%%%%%%%%%%%%%

\section{Introduction}

Symmetries play an important role both in physics and mathematics.
They are described by transformations leaving structural relations
unchangeable. Their importance range from fundamental and
theoretical aspects to concrete applications, having profound
implications in the dynamical behavior of the systems and in their
basic qualitative properties. Knowledge of symmetries result on a
deep insight about the inner structure of a system, and permits to
apply the conservation laws to the investigation of the objects,
\textrm{i.e.} to link the invariance principles with the
conservation laws. This interrelation includes three classes of
the most fundamental principles of physics: symmetry,
conservation, and extremality.

When a closed system is characterized by a quantity which remains
unchangeable in the course of time, no matter what kind of
processes take place in the system, such quantity is said to be a
conservation law. Some fundamental conservation laws include the
conservation of energy, impulse, momentum impulse, motion of the
centre of gravity, electrical charge, and others. All physical
laws are described in terms of differential equations (equations
of motion). The conservation laws represent the first integrals of
the equations of motion and are important for three reasons.
Firstly, the task of solving the equations of motion explicitly is
not always possible and knowledge of the first integrals may
considerably simplify that task. Secondly, often there is no more
necessity to solve the equation of motion, as the useful
information is contained in the conservation laws. Thirdly, the
conservation laws have a deep physical meaning and can be measured
directly.

All basic differential equations of physics (\textrm{i.e.} the
equations of motion of physical systems) have a variational
structure. In other words, the equations of motion of a physical
system are the Euler-Lagrange equations of a certain variational
problem. It turns out that the conservation laws are the result of
the invariance of the action with respect to a continuous group of
transformations, given by some symmetry principle. The more
general expression of the interrelation symmetry/variational
structure/conservation, is given by Noether's theorem. Noether's
theorem asserts that the conservation laws for a system of
differential equations which correspond to the Euler-Lagrange
equations of a certain variational problem, come from the
invariance of the variational functional with respect to a
one-parameter continuous group of transformations. The group of
symmetry transformations requested by Noether's theorem depend,
of course, on the physical properties of the system.
We refer the interested reader to \cite{alik}.

Conservative physical systems imply frictionless motion and are a
simplification of the real dynamical world. Almost all systems
contain internal damping and are subject to external forces. For
non-conservative dynamical systems, \textrm{i.e.} in the presence
of non-conservative forces (forces that do not store energy and
which are not equivalent to the gradient of a potential), the
conservations law are broken so that the standard Lagrangian or
Hamiltonian formalism is no longer valid for describing the
behavior of the system. Methodologically, Newtonian dissipative
dynamical systems are a complement to conservative systems,
because not only energy, but also other physical quantities as
linear or angular momentums, are not conserved. In this case the
classical Noether's theorem ceases to be valid. However, it is
still possible to obtain a Noether-type theorem which covers both
conservative (closed system) and nonconservative cases
\cite{CD:Djukic:1980,GastaoIJTS}. Roughly speaking, one can prove
that Noether's conservation laws are still valid if a new term,
involving the nonconservative forces, is added to the standard
constants of motion.

In order to better model non-conservative dynamical systems,
a novel approach entitled Fractional Action-Like
Variational Approach (FALVA) has been recently introduced
\cite{CD:El-Na:2005,El-Nabulsi2005a}. This
approach is based on the concept of fractional integration.
Fractional theory plays an important role in the understanding of
both conservative and non-conservative behaviors of complex
dynamical systems, and has important physical applications in
various fields of science, \textrm{e.g.} physics, material sciences,
chemistry, biology, scaling phenomena, etc.
In \cite{CD:El-Na:2005,El-Nabulsi2005a} Riemann-Liouville
fractional integral functionals, depending on a parameter $\alpha$
but not on fractional-order derivatives of order $\alpha$, are
introduced and respective fractional Euler-Lagrange type equations
obtained. In \cite{CD:Jumarie:2007}, Jumarie uses the variational
calculus of fractional order to derive an Hamilton-Jacobi equation
and a Lagrangian variational approach to the optimal control of
one-dimensional fractional dynamics with fractional cost functional.
In this paper we extend the results of \cite{CD:Gastao:2006}
to more general FALVA problems with higher-order derivatives.

%%%%%%%%%%%%%%%%%%%%%%%%%%%%%%%%%%%%%%

\section{Preliminaries}

We begin by collecting the necessary results from
\cite{CD:El-Na:2005,CD:GasDel:2007}. In 2005 El-Nabulsi
(\textrm{cf.} \cite{CD:El-Na:2005}) introduced the FALVA problem
as follows:

\begin{problem}
\label{pb:FAL} Find the stationary points of the integral
functional
\begin{equation}
\label{Pi} I[q(\cdot)] = \frac{1}{\Gamma(\alpha)}\int_a^t
L\left(\theta,q(\theta),\dot{q}(\theta)\right)(t-\theta)^{\alpha-1}
d\theta 
\end{equation}
under the initial condition $q(a)=q_{a}$, where
$\dot{q} = \frac{dq}{d\theta}$, $\Gamma$ is the Euler gamma function,
$0<\alpha\leq 1$, $\theta$ is the intrinsic time, $t$ is the observer time,
$t\neq\theta$, and the Lagrangian $L :[a,b] \times
\mathbb{R}^{n} \times \mathbb{R}^{n} \rightarrow \mathbb{R}$ is a
$C^{2}$-function with respect to all its arguments.
\end{problem}

Along all the work we denote by $\partial_{i}L$, $i=1, 2, 3$, the
partial derivative of $L(\cdot,\cdot,\cdot)$ with respect to its
$i$th argument.

Theorem~\ref{Thm:NonDtELeq} summarizes one of the main results of
\cite{CD:El-Na:2005}.

\begin{theorem}[\cite{CD:El-Na:2005}]
\label{Thm:NonDtELeq} If $q(\cdot)$ is a solution to
Problem~\ref{pb:FAL} (\textrm{i.e.}, $q(\cdot)$ is a critical
point of the functional \eqref{Pi}), then $q(\cdot)$ satisfy the
following \emph{Euler-Lagrange equation}:
\begin{equation}
\label{eq:elif}
\partial_{2} L\left(\theta,q(\theta),\dot{q}(\theta)\right)-\frac{d}{d\theta}
\partial_{3} L\left(\theta,q(\theta),\dot{q}(t)\right)=
\frac{1-\alpha}{t-\theta}\partial_{3}
L\left(\theta,q(\theta),\dot{q}(\theta)\right)\, .
\end{equation}
\end{theorem}

In \cite{CD:GasDel:2007} the authors introduced the following
FALVA problem with higher-order derivatives:

\begin{problem}
\label{pb:FALOrdS} Find the stationary points of the integral
functional
\begin{equation}
\label{Pm} I^{m}[q(\cdot)] = \frac{1}{\Gamma(\alpha)} \int_{a}^t
L\left(\theta,q(\theta),\dot{q}(\theta),\ldots,q^{(m)}(\theta)\right)(t-\theta)^{\alpha-1}
d\theta \, , 
\end{equation}
$m\geq 1$, under the initial conditions
\begin{equation}
\label{eq:Pm} q^{(i)}(a)=q^{i}_{a}\, , \quad i=0,\ldots,m \, ,
\end{equation}
where $q^{0}(\theta)=q(\theta)$, $q^{(i)}(\theta)$ is the
derivative of $q(\theta)$ of order $i$, $\Gamma$ is the Euler
gamma function, $0<\alpha\leq 1$, $\theta$ is the intrinsic time,
$t$ is the observer time, $t\neq\theta$, and the Lagrangian $L
:[a,b] \times \mathbb{R}^{n\times(m+1)} \rightarrow \mathbb{R}$ is
a function of class $C^{2m}$ with respect to all its arguments.
\end{problem}

\begin{remark}
In the particular case where $m=1$, Problem~\ref{pb:FALOrdS}
reduces to Problem~\ref{pb:FAL}.
\end{remark}

Theorem~\ref{Thm:ELdeordm} generalizes Theorem~\ref{Thm:NonDtELeq}
to the higher-order case.

\begin{theorem}[\cite{CD:GasDel:2007}]
\label{Thm:ELdeordm} If $q(\cdot)$ is a stationary point of
\eqref{Pm}, then $q(\cdot)$ satisfy the following
\emph{higher-order Euler-Lagrange equation}:
\begin{equation}
 \label{eq:ELdeordm}
 \sum_{i=0}^{m}(-1)^{i}\frac{d^{i}}{d\theta^{i}}\partial_{i+2}
L\left(\theta,q(\theta),\dot{q}(\theta),\ldots,q^{(m)}(\theta)\right)
=F\left(\theta,q(\theta),\dot{q}(\theta),\ldots,q^{(2m-1)}(\theta)\right)\, ,
\end{equation}
$m\geq 1$, where
\begin{multline}\label{eq:FNCFA}
F\left(\theta,q(\theta),\dot{q}(\theta),\ldots,q^{(2m-1)}(\theta)\right)
=\frac{1-\alpha}{t-\theta}\sum_{i=1}^{m}i(-1)^{i-1}\frac{d^{i-1}}{d\theta^{i-1}}\,\partial_{i+2}
L\left(\theta,q(\theta),\dot{q}(\theta),\ldots,q^{(m)}(\theta)\right)\\
+\sum_{k=2}^{m}\sum_{i=2}^{k}(-1)^{i-1}
\frac{\Gamma(i-\alpha+1)}{(t-\theta)^{i}\Gamma(1-\alpha)}\,
{k\choose k-i}\frac{d^{k-i}}{d\theta^{k-i}}\,
\partial_{k+2}L\left(\theta,q(\theta),\dot{q}(\theta),\ldots,q^{(m)}(\theta)\right)\, .
\end{multline}
\end{theorem}

We borrow from \cite{CD:JMS:Torres:2004} the notation
\begin{equation}
\label{eq:eqprin}
\psi^{j}=\sum_{i=0}^{m-j}(-1)^{i}\frac{d^{i}}{d\theta^{i}}\partial_{i+j+2}
L\left(\theta,q(\theta),\dot{q}(\theta),\ldots,q^{(m)}(\theta)\right)
\, , \quad j=1,\ldots,m \, ,
\end{equation}
which is useful for our purposes because of the following property:
\begin{equation}
\label{eq:eqprin1} \frac{d}{d\theta}\psi^{j}=\partial_{j+1}
L\left(\theta,q(\theta),\dot{q}(\theta),\ldots,q^{(m)}(\theta)\right)-\psi^{j-1}
\, , \quad j=1,\ldots,m \, .
\end{equation}

\begin{remark}
One can write equations \eqref{eq:ELdeordm} in the following form:
\begin{equation}
\label{eq:ELdeordm1}
\partial_{2}
L\left(\theta,q(\theta),\dot{q}(\theta),\ldots,q^{(m)}(\theta)\right)
-\frac{d}{d\theta}\psi^{1}
=F\left(\theta,q(\theta),\dot{q}(\theta),\ldots,q^{(2m-1)}(\theta)\right)\, .
\end{equation}
\end{remark}

\begin{theorem}[\cite{CD:GasDel:2007}]
\label{theo:cDRifm} If $q(\cdot)$ is a solution
of Problema~\ref{pb:FALOrdS}, then it satisfy the following
higher-order DuBois-Raymond condition:
\begin{multline}
\label{eq:DBRordm}
\frac{d}{d\theta}\left\{L\left(\theta,q(\theta),\dot{q}(\theta),\ldots,q^{(m)}(\theta)\right)
-\sum_{j=1}^{m}\psi^{j}\cdot q^{(j)}(\theta)\right\}\\
=\partial_{1}
L\left(\theta,q(\theta),\dot{q}(\theta),\ldots,q^{(m)}(\theta)\right)
+F\left(\theta,q(\theta),\dot{q}(\theta),\ldots,q^{(2m-1)}(\theta)\right)\cdot
\dot{q}(\theta)\, ,
\end{multline}
where $F$ and $\psi^j$ are defined as in \eqref{eq:FNCFA} and
\eqref{eq:eqprin} respectively.
\end{theorem}

%%%%%%%%%%%%%%%%%%%%%%%%%%%%%%%%%%%%%%

\section{Main Result}

In this work we generalize the Noether-type theorem proved in \cite{CD:Gastao:2006}
to the more general FALVA problem with higher-order derivatives.

\subsection{Noether's theorem for higher-order FALVA problems}

In order to generalize the Noether's theorem to
Problem~\ref{pb:FALOrdS} (see Theorem~\ref{theo:tnifm} below) we
use the DuBois-Reymond necessary stationary condition
\eqref{eq:DBRordm} and the following invariance definition.

\begin{definition} (Invariance of \eqref{Pm})
\label{def:invaifm} The functional \eqref{Pm} is said to be invariant under
the infinitesimal transformations
\begin{equation}
\label{eq:tinfif}
\begin{cases}
\bar{\theta} = \theta + \varepsilon\tau(\theta,q) + o(\varepsilon) \\
\bar{q}(\bar{\theta}) = q(\theta) + \varepsilon\xi(\theta,q) + o(\varepsilon) \\
\end{cases}
\end{equation}
if
\begin{multline}
\label{eq:invifm}
L\left(\bar{\theta},\bar{q}(\bar{\theta}),{\bar{q}}'(\bar{\theta}),\ldots,\bar{q}'^{(m)}(\bar{\theta})\right)
(t-\bar{\theta})^{\alpha-1}\frac{d\bar{\theta}}{d \theta}\\
= L\left(\theta,q(\theta),\dot{q}(\theta),\ldots,q^{(m)}(\theta)\right)(t-\theta)^{\alpha-1}
\\+ \varepsilon
(t-\theta)^{\alpha-1}\frac{d\Lambda}{d\theta}\left(\theta,q(\theta),\dot{q}(\theta),\ldots,q^{(2m-1)}(\theta)\right)
+ o(\varepsilon)\, .
\end{multline}
\end{definition}

\begin{remark}
Expressions $\bar{q}'^{(i)}$ in equation \eqref{eq:invifm},
$i=1,\ldots,m$, are interpreted as
\begin{equation}
\label{eq:invifm1}
\bar{q}'=\frac{d\bar{q}}{d\bar{\theta}}=\frac{\frac{d\bar{q}}{d\theta}}{\frac{d\bar{\theta}}{d\theta}}\,\,,
\quad \bar{q}'^{(i)}=\frac{d^{i}\bar{q}}{d\bar{\theta}^{i}}=
\frac{\frac{d}{d\theta}\left(\frac{d^{i-1}}{d\bar{\theta}^{i-1}}\bar{q}\right)}{\frac{d\bar{\theta}}{d\theta}}
\quad (i=2,\ldots,m) \, .
\end{equation}
\end{remark}

Next theorem gives a necessary and sufficient
condition for invariance of \eqref{Pm}. Theorem~\ref{thm:cnsi} is useful
to check invariance and also to compute the infinitesimal generators
$\tau$ and $\xi$.

\begin{theorem} (Necessary and sufficient condition for invariance of
\eqref{Pm}) \label{thm:cnsi}
The integral functional \eqref{Pm} is
invariant in the sense of Definition~\ref{def:invaifm} if and only if
\begin{multline}
\label{eq:cnsiifm}
\partial_{1}
L\left(\theta,q(\theta),\dot{q}(\theta),\ldots,q^{(m)}(\theta)\right)\tau
+\sum_{i=0}^{m}\partial_{i+2}
L\left(\theta,q(\theta),\dot{q}(\theta),\ldots,q^{(m)}(\theta)\right)\cdot
\rho^{i}
\\
+ L\left(\theta,q(\theta),\dot{q}(\theta),\ldots,q^{(m)}(\theta)\right)
\left( \dot{\tau} + \frac{1-\alpha}{t-\theta} \tau \right)
=\dot{\Lambda}\left(\theta,q(\theta),\dot{q}(\theta),\ldots,q^{(2m-1)}(\theta)\right)\,
,
\end{multline}
where
\begin{equation}
\label{eq:cnsiifm1}
\begin{cases}
\rho^{0}=\xi \, , \\
\rho^{i}=\frac{d}{d\theta}\left(\rho^{i-1}\right)-q^{(i)}(\theta)\dot{\tau}
\, ,\quad i=1,\ldots,m \, .
\end{cases}
\end{equation}
\end{theorem}

\begin{remark}
If $\alpha=1$, condition \eqref{eq:cnsiifm} gives the higher-order
necessary and sufficient condition of invariance proved in
\cite{CD:JMS:Torres:2004}:
\begin{multline*}
\partial_{1}
L\left(\theta,q(\theta),\dot{q}(\theta),\ldots,q^{(m)}(\theta)\right)\tau
+\sum_{i=0}^{m}\partial_{i+2}
L\left(\theta,q(\theta),\dot{q}(\theta),\ldots,q^{(m)}(\theta)\right)\cdot
\rho^{i}
\\
+ L\left(\theta,q(\theta),\dot{q}(\theta),\ldots,q^{(m)}(\theta)\right)
 \dot{\tau}
=\dot{\Lambda}\left(\theta,q(\theta),\dot{q}(\theta),\ldots,q^{(2m-1)}(\theta)\right)\, .
\end{multline*}
\end{remark}

\begin{proof} (of Theorem~\ref{thm:cnsi})
Differentiating equation \eqref{eq:invifm} with respect to $\varepsilon$, then
setting $\varepsilon=0$, we obtain:
\begin{equation*}
\partial_{1} L\tau+\sum_{i=0}^{m}\partial_{i+2} L\cdot\frac{\partial}{\partial \varepsilon}
\left.\left(\frac{d^{i}\bar{q}}{d\bar{\theta}^{i}}\right)\right|_{\varepsilon=0}
+L\left( \dot{\tau} + \frac{1-\alpha}{t-\theta} \tau \right)
=\dot{\Lambda} \, .
\end{equation*}
The intended conclusion follows from \eqref{eq:invifm1}:
\begin{equation*}
\frac{\partial}{\partial\varepsilon}\left.\left(\frac{d\bar{q}}{d\bar{\theta}}\right)\right|_{\varepsilon=0}
=\dot{\xi}-\dot{q}\dot{\tau} \, ,
\end{equation*}
\begin{equation*}
\frac{\partial}{\partial\varepsilon}\left.\left(\frac{d^{i}\bar{q}}{d\bar{\theta}^{i}}\right)\right|_{\varepsilon=0}
=\frac{d}{d\theta}\left[\frac{\partial}{\partial\varepsilon}
\left.\left(\frac{d^{i-1}\bar{q}}{d\bar{\theta}^{i-1}}\right)\right|_{\varepsilon=0}\right]
-q^{(i)}\dot{\tau}\, , \quad i=2,\ldots,m \, .
\end{equation*}
\end{proof}

\begin{definition} (Higher-order conservation law)
\label{def:leicoifm} A quantity
$C\left(\theta,q(\theta),\dot{q}(\theta),\ldots,q^{(2m-1)}(\theta)\right)$
is said to be a \emph{conservation law} if
\begin{equation*}
\frac{d}{d\theta}C\left(\theta,q(\theta),\dot{q}(\theta),\ldots,q^{(2m-1)}(\theta)\right)=0
\end{equation*}
along all the solutions $q(\cdot)$ of the higher-order
Euler-Lagrange equation \eqref{eq:ELdeordm}.
\end{definition}

\begin{theorem} (Higher-order Noether's theorem)
\label{theo:tnifm} If the integral functional \eqref{Pm} is
invariant in the sense of Definition~\ref{def:invaifm} and
$\tau(\theta,q)$ and $\xi(\theta,q)$ satisfy the condition
\begin{equation}
\label{eq:confifm}
G\left(\theta,q(\theta),\dot{q}(\theta),\ldots,q^{(2m-1)}(\theta)\right)\cdot
\Omega=
-L\left(\theta,q(\theta),\dot{q}(\theta),\ldots,q^{(m)}(\theta)\right)\tau\, ,
\end{equation}
where
\begin{multline}\label{eq:GFA}
G\left(\theta,q(\theta),\dot{q}(\theta),\ldots,q^{(2m-1)}(\theta)\right)
=\sum_{i=1}^{m}(-1)^{i-1}i\frac{d^{i-1}}{d\theta^{i-1}}\partial_{i+2}
L\left(\theta,q(\theta),\dot{q}(\theta),\ldots,q^{(m)}(\theta)\right)\\
+\sum_{k=2}^{m}\sum_{i=2}^{k}(-1)^{i}
 \frac{\Gamma(i-\alpha+1)}{\Gamma(2-\alpha)(t-\theta)^{i-1}}
 {k\choose k-i}\frac{d^{k-i}}{d\theta^{k-i}}
 \partial_{k+2}L\left(\theta,q(\theta),\dot{q}(\theta),\ldots,q^{(m)}(\theta)\right)
\end{multline}
and $\Omega=\xi-\dot{q}\tau$, then
\begin{multline}
\label{eq:TeNetm}
 C\left(\theta,q(\theta),\dot{q}(\theta),\ldots,q^{(2m-1)}(\theta)\right) \\
 = \sum_{j=1}^{m}\psi^{j}\cdot
\rho^{j-1}+\left(L\left(\theta,q(\theta),\dot{q}(\theta),\ldots,q^{(m)}(\theta)\right)
-\sum_{j=1}^{m}\psi^{j}\cdot q^{(j)}(\theta)\right)\tau\\
-\Lambda\left(\theta,q(\theta),\dot{q}(\theta),\ldots,q^{(2m-1)}(\theta)\right)
\end{multline}
is a conservation law.
\end{theorem}

\begin{remark}
Under hypothesis \eqref{eq:confifm}, the necessary and sufficient condition
of invariance \eqref{eq:cnsiifm} takes the following form:
\begin{multline}
\label{eq:cnsiifm2}
\partial_{1}
L\left(\theta,q(\theta),\dot{q}(\theta),\ldots,q^{(m)}(\theta)\right)\tau
+\sum_{i=0}^{m}\partial_{i+2}
L\left(\theta,q(\theta),\dot{q}(\theta),\ldots,q^{(m)}(\theta)\right)\cdot
\rho^{i}
\\
+ L\left(\theta,q(\theta),\dot{q}(\theta),\ldots,q^{(m)}(\theta)\right)
 \dot{\tau}-F\left(\theta,q(\theta),\dot{q}(\theta),\ldots,q^{(2m-1)}(\theta)\right)\cdot
\Omega\\
=\dot{\Lambda}\left(\theta,q(\theta),\dot{q}(\theta),\ldots,q^{(2m-1)}(\theta)\right) \, .
\end{multline}
\end{remark}

\begin{proof} (of Theorem~\ref{theo:tnifm})
We begin by writing the Noether's conservation law
\eqref{eq:TeNetm} in the form
\begin{equation}
\label{eq:TeNetm1}
 C =\psi^1\cdot\rho^0+
\sum_{j=2}^{m}\psi^{j}\cdot \rho^{j-1}+\left(L
-\sum_{j=1}^{m}\psi^{j}\cdot q^{(j)}(\theta)\right)\tau -\Lambda \, .
\end{equation}
Differentiation of equation \eqref{eq:TeNetm1} with respect to
$\theta$ gives
\begin{multline}
\label{eq:TeNetm2} \dot{\Lambda}=\rho^0\cdot\frac{d}{d\theta}\psi^1
+\psi^1\cdot\frac{d}{d\theta}\rho^0
+\sum_{j=2}^{m}\left(\rho^{j-1}\cdot\frac{d}{d\theta}\psi^{j}
+\psi^{j}\cdot\frac{d}{d\theta}\left(\rho^{j-1}\right)\right
)\\+\tau\frac{d}{d\theta}\left(L-\sum_{j=1}^{m}\psi^{j}\cdot
q^{(j)}(\theta)\right)+\left(L-\sum_{j=1}^{m}\psi^{j}\cdot
q^{(j)}(\theta)\right)\frac{d}{d\theta}\tau \, .
\end{multline}
Using the Euler-Lagrange equation \eqref{eq:ELdeordm}, the
DuBois-Reymond condition \eqref{eq:DBRordm}, and relations
\eqref{eq:eqprin1}  and \eqref{eq:cnsiifm1} in \eqref{eq:TeNetm2},
we obtain:
\begin{multline}
\label{eq:dems}
 \dot{\Lambda}=\left(\partial_{2}
L-F\right)\cdot\xi+\psi^{1}\cdot(\rho^1+\dot{q}\dot{\tau})
+\sum_{j=2}^{m}\left[\left(\partial_{j+1}
L-\psi^{j-1}\right)\cdot\rho^{j-1}+
\psi^{j}\cdot\left(\rho^{j}+q^{(j)}(\theta)\dot{\tau}\right)\right]\\
+\left(\partial_{1} L+F\cdot
\dot{q}\right)\tau+\left(L-\sum_{j=1}^{m}\psi^{j}\cdot
q^{(j)}(\theta)\right)\dot{\tau}\\
=\partial_{1} L \tau+L\dot{\tau}+\partial_{2}
L\cdot\xi+\psi^{1}\cdot(\rho^1+\dot{q}\dot{\tau})
-\psi^1\cdot\rho^1\\-\psi^1\cdot\dot{q}\dot{\tau}+\psi^m\cdot\rho^m+\sum_{j=2}^{m}\partial_{j+1}
L\cdot\rho^{j-1}\, .
\end{multline}
Simplification of \eqref{eq:dems} lead us to the
necessary and sufficient condition of invariance \eqref{eq:cnsiifm2}.
\end{proof}

In the particular case $m = 1$ we obtain from our Theorem~\ref{theo:tnifm}
the main result of \cite{CD:Gastao:2006}.

\begin{corollary} (\textrm{cf.} \cite{CD:Gastao:2006})
If the integral functional \eqref{Pi} is invariant under the
infinitesimal transformations \eqref{eq:tinfif}, and
$\tau(\theta,q)$ and $\xi(\theta,q)$ satisfy the condition
\begin{equation}
\label{eq:condif}
\partial_{3}
L\left(\theta,q,\dot{q}\right)\cdot\Omega=-L\left(\theta,q,\dot{q}\right)\tau\, ,
\end{equation}
then
\begin{equation}
\label{eq:TeNetFA}
 C(\theta,q,\dot{q}) =
\partial_{3} L\left(\theta,q,\dot{q}\right)\cdot\xi(\theta,q)
+ \left( L(\theta,q,\dot{q}) - \partial_{3}
L\left(\theta,q,\dot{q}\right) \cdot \dot{q} \right)
\tau(\theta,q)-\Lambda\left(\theta,q,\dot{q}\right)
\end{equation}
is a conservation law (\textrm{i.e.}, \eqref{eq:TeNetFA} is constant
along all the solutions $q(\cdot)$ of the Euler-Lagrange equation \eqref{eq:elif}).
\end{corollary}

\begin{proof}
For $m=1$ we obtain from \eqref{eq:GFA} and \eqref{eq:TeNetm} that
\begin{equation}\label{eq:GFA1}
    G\left(\theta,q,\dot{q}\right)=\partial_3 L\left(\theta,q,\dot{q}\right)
\end{equation}
and
\begin{equation}\label{eq:GFA2}
     C(\theta,q,\dot{q}) =\psi^1\cdot\rho^0+\left(L(\theta,q,\dot{q})
     -\psi^1\cdot\dot{q}(\theta)\right)\tau-\Lambda(\theta,q,\dot{q}) \, .
     \end{equation}
Having in mind the equations \eqref{eq:eqprin} and \eqref{eq:cnsiifm1},
we conclude that
\begin{equation}\label{eq:GFA3}
    \begin{cases}
    \psi^1=\partial_3 L\left(\theta,q,\dot{q}\right)\,,\\
    \rho^0=\xi\, .
    \end{cases}
\end{equation}
We obtain the intended result substituting \eqref{eq:GFA3} into
\eqref{eq:GFA2}.
\end{proof}

%%%%%%%%%%%%%%%%%%%%%%%%%%%%%%%%%%%%%%%%%%%%%%%%%%%%%%%%%%%%%%

\subsection{Example}

In order to illustrate our result, we consider an example for
which the Lagrangian $L$ do not depend explicitly on the intrinsic
time $\theta$.

\begin{example}
\label{FA2} Let us consider the following second-order ($m=2)$
FALVA problem: to find a stationary function $q(\cdot)$ for the
integral functional
\begin{equation}
\label{eq:exFA2}
I^2[q(\cdot)]=\frac{1}{2}\int_0^t\left(aq^2+b\dot{q}^2
+\ddot{q}^2\right)(t-\theta)^{\alpha-1} d\theta\, ,
\end{equation}
where $a$ and $b$ are arbitrary constants. In this case the
Euler-Lagrange equation \eqref{eq:ELdeordm} reads
\begin{equation}
\label{eq:ELdeord2} -aq +\frac{b(1-\alpha)}{t-\theta}\dot{q}
+\left(b-\frac{(1-\alpha)(2-\alpha)}{(t-\theta)^2}\right)\ddot{q}
-\left(1+\frac{2(1-\alpha)}{(t-\theta)}\right)\dddot{q}=0\, .
\end{equation}
Since the Lagrangian $L$ do not depend explicitly on the
independent variable $\theta$, the necessary and sufficient
invariance condition \eqref{eq:cnsiifm2} is satisfied with
\begin{gather}
\tau=1 \label{eq:exFA4}\,,\\
\xi=0 \label{eq:exFA5}\, ,\\
\dot{\Lambda}=F\dot{q} \Rightarrow \Lambda = \int F \dot{q} \, d
\theta\, , \label{eq:exFA3}
\end{gather}
where
\begin{equation}
\label{eq:Ex:F} F=\frac{1-\alpha}{t-\theta}(b\dot{q}-2\dddot{q})
-\frac{(1-\alpha)(2-\alpha)}{(t-\theta)^2}\ddot{q}\,.
\end{equation}
The conservation law \eqref{eq:TeNetm} with $m=2$
takes the following form:
\begin{multline}
\label{eq:exFA7}
 C(\theta,q,\dot{q},\ddot{q},\dddot{q}) =
L\left(\theta,q,\dot{q},\ddot{q}\right)\tau +\left(\partial_{3}
L\left(\theta,q,\dot{q},\ddot{q}\right)
-\frac{d}{d\theta}\partial_{4}
L\left(\theta,q,\dot{q},\ddot{q}\right)\right)\cdot\Omega
\\+\partial_{4}
L\left(\theta,q,\dot{q},\ddot{q}\right)\cdot\dot{\Omega}
-\Lambda\left(\theta,q,\dot{q},\ddot{q},\dddot{q}\right) \, .
\end{multline}
Substituting the quantities $L=\frac{1}{2}\left(aq^2+b\dot{q}^2
+\ddot{q}^2\right)$, \eqref{eq:exFA4}, \eqref{eq:exFA5} and
\eqref{eq:exFA3} into \eqref{eq:exFA7}, we conclude that
\begin{equation}
\label{eq:GCC} \frac{1}{2}\left(aq^2-b\dot{q}^2
+3\ddot{q}^2\right) - \dot{q} \dddot{q} -\int F \dot{q} \, d
\theta
\end{equation}
is constant along any solution $q$ of \eqref{eq:ELdeord2}.
\end{example}

If $\alpha=1$, then one see from \eqref{eq:Ex:F} that $F=0$, and
\eqref{eq:GCC} gives the classical result in
\cite{CD:Djukic:1980}:
\begin{equation*}
\frac{1}{2}\left(aq^2
-b\dot{q}^2+3\ddot{q}^2\right)-\dot{q}\,\dddot{q}
\end{equation*}
is a conservation law.

%%%%%%%%%%%%%%%%%%%%%%%%%%%%%%%%%%%%%%%%%%%%%%%%%%

\section*{Acknowledgments}

This work is part of the first author's PhD project, partially
supported by the \emph{Portuguese Institute for Development}
(IPAD). The authors are also grateful to the support of the
\emph{Portuguese Foundation for Science and Technology} (FCT)
through the \emph{Centre for Research in Optimization and Control}
(CEOC) of the University of Aveiro, cofinanced by the European
Community Fund FEDER/POCI 2010.

%%%%%%%%%%%%%%%%%%%%%%%%%%%%%%%%%%%%%%%%%%%%%%%%%%%%%%%

%%%%%%%%%%%%%%%%%%%%%%%%%%%%%%%%%%%%%%%%%%%%%%%%%%%%%%%

\end{document}